\documentclass[11pt]{amsart}

\usepackage{amsmath}
\usepackage{amssymb}
\usepackage{graphicx}
\usepackage{graphicx,psfrag,epsfig}

\usepackage{amsfonts}
\usepackage{color}

\newtheorem{lemma}{Lemma}[section]
\newtheorem{theorem}[lemma]{Theorem}

\theoremstyle{definition}

\newtheorem{remark}[lemma]{Remark}

\def\de{\overset{\mathrm{def}}{=}}
\def\tom{{\mathrm{w}}}
\def\TO{{\tilde{\Omega}}}

\numberwithin{equation}{section}

\newcommand{\DS}{\displaystyle}
\newcommand{\ignore}[1]{}

\def\eps{{\varepsilon}}

\def\Card{{\rm Card}}
\def\Const{{\rm Const}}

\def\Prob{{\mathbb{P}}}

\def\EXP{{\mathbb{E}}}

\def\bbP{\mathbb{P}}

\def\bbS{\mathbb{S}}

\def\naturals{\mathbb{N}}

\def\reals{\mathbb{R}}

\def\integers{\mathbb{Z}}

\def\bD{\mathbf{D}}

\def\bE{\mathbf{E}}

\def\bP{\mathbf{P}}

\def\bQ{\mathbf{Q}}

\def\blambda{\boldsymbol{\lambda}}

\def\brC{{\bar C}}

\def\brD{{\bar D}}

\def\brR{{\bar R}}

\def\brk{{\bar k}}

\def\cO{\mathcal{O}}

\def\cT{\mathcal{T}}

\def\cX{\mathcal{X}}

\def\cZ{\mathcal{Z}}

\def\fa{\mathfrak{a}}

\def\fF{\mathfrak{F}}

\def\fG{\mathfrak{G}}

\def\fH{\mathfrak{H}}

\def\fL{\mathfrak{L}}

\def\fg{\mathfrak{g}}

\def\fl{\mathfrak{l}}

\def\ft{\mathfrak{t}}

\def\hD{{\hat D}}

\def\hk{{\hat k}}

\def\hu{{\hat u}}

\def\hz{{\hat z}}

\def\trho{{\tilde\rho}}

\def\txi{{\tilde\xi}}

\makeatother

\def\beq{\begin{equation}}
\def\eeq{\end{equation}}

\title[Local Limit Theorems for RWRE on a Strip]
{Local Limit Theorems for Random Walks in a Random Environment on a Strip}

\author[D. Dolgopyat]{Dmitry Dolgopyat}
\address[D. Dolgopyat]{Department of
Mathematics and Institute of Physical Science and Technology\\
University  of Maryland\\
College Park, MD 20742\\
USA}
\email{{\tt dmitry@math.umd.edu}}
\author[I. Goldsheid]{Ilya Goldsheid}
\address[I. Goldsheid]{School of Mathematical Sciences\\
Queen Mary University of London\\
Mile End Road\\
London E1 4NS\\
Great Britain}
\email{{\tt I.Goldsheid@qmul.ac.uk}}


\begin{document}
\maketitle

\begin{abstract}
The paper consists of two parts.
In the first part we review recent work on limit theorems for
random walks in random environment (RWRE) on a strip with jumps to the
nearest layers.
In the second part, we prove the quenched Local Limit Theorem (LLT) for
the position of the walk in the transient diffusive regime. This fills an important gap
in the literature.
We then obtain two corollaries of the quenched LLT.
The first one is the annealed version of the LLT on a strip. The second one is the proof of
the fact that the distribution of the environment viewed from the particle (EVFP) has a limit
for a. e. environment. In the case of the random walk with jumps to nearest neighbours in dimension one,
the latter result is a theorem of Lally \cite{L}. Since the strip model incorporates the walks with
bounded jumps on a one-dimensional lattice, the second corollary also solves the long
standing problem of extending Lalley's result to this case.


\end{abstract}

\section{Introduction.}

This paper is devoted to the local limiting behavior of a random walk in random environment
on a strip. The study of RWRE was initiated in \cite{So}. The annealed limit theorems for one dimensional
RWRE go back to \cite{KKS} who have analyzed the transient case.  In the recurrent case, Sinai proved in \cite{S1}
a surprising fact: the correct scaling in the limit theorem is $\ln^2 t.$

Results describing the behavior of RWs in quenched (frozen) random environments are relatively recent.
The quenched CLT for the simple transient RWRE in the diffusive regime
\footnote
{The precise definition of the
diffusivity condition will be explained later (cf. Theorem \ref{ThAnn}(c)).}
was proved in \cite{G} in 2007 for a wide class of environments (including the iid case)
and, independently, for iid environments in \cite{P}.

It should be emphasized that in the diffusive regime the CLT holds for almost all environments and
that there is a drastic difference between the diffusive and subdiffusive regime. Namely, it has been
realized in \cite{P2, PZ} that in the subdiffusive regime, for almost every environment, the
simple RW does not have a distributional limit. In fact, the quenched distribution of the walk
turns out to be quite non-trivial, see \cite{DG1, PS1, PS2} for detailed discussion of this problem
(related results were obtained independently in \cite{ESTZ}).

The papers \cite{So, KKS} as well as \cite{S1} relied heavily on the fact that the walk is on a one-dimensional lattice
and can jump only to the nearest neighbors (the so called simple walk). Therefore, the natural
question asked by Sinai in \cite{S1} was how to extend the results about nearest neighbor RWRE on the line
to walks with bounded jumps. The key tools needed to achieve such a generalization were developed in
\cite{BG1} which introduced a more general model, namely the RWRE on a strip. This model includes the RWRE
with bounded jumps on a line as a special case. The results of \cite{BG1} were instrumental for proving
limit theorems for RWRE on the strip which was done in \cite{G1} (transient diffusive walks),
\cite{BG2} (recurrent walks), and \cite{DG2} (transient subdiffusive walks).

The next step was to prove the local limit theorem in the diffusive regime.
In the case of the simple RWRE  the quenched LLT was proved in
 \cite{DG3}. We would like to mention two corollaries of this result: it implies the annealed LLT, and
the LLT is the main ingredient in a new proof
of a result of Lalley stating the existence of the limit of
the distribution of the environment viewed from the particle (EVFP).
Clearly, both the LLT and the EVFP measure characterize the local limiting behaviour of a
RW but the strong connection between the two is a newly discovered phenomena.
Similar results for recurrent diffusive walks were obtained in \cite{DG4}.


The aim of this paper is twofold. First, we review the recent results about RWRE on the strip
placing a particular emphasis on the LLT and on mixing of the EVFP process.
Secondly, we extend the main results from \cite{DG3} to the case of the strip.
We prove that in the transient diffusive regime the quenched LLT holds for almost every (a.e.) environment
(Theorem \ref{ThQLLTstrip}). It may be useful to emphasized that,
unlike in the CLT, one has to have an additional random factor $\rho_n$ in front of the exponent which is due to
the randomness of the environment.
As in \cite{DG3}, the quenched LLT implies the annealed one
(Theorem \ref{ThAnnLLTs}).

Finally we prove that in the diffusive regime the limit of the distribution of the
environment viewed from the particle exists for a. e. environment (Theorem
\ref{ThEPStrip}). Our proof of this result is completely
different from the one given in \cite{L}.

Since the study of random walks with bounded jumps in dimension one can be reduced to those
on a strip with jumps to nearest layers (see \cite{BG1, G1}), this solves
the problem of extending the Lalley's result to the walks on $\mathbb{Z}$ with bounded jumps.

The layout of the paper is the following.
Section \ref{sec8} contains a review of RWRE on the strip. We first introduce the main tools needed for
analysis of these walks and then describe the limit theorems in that setting.
The new results pertaining to diffusive transient walks are proved in Sections \ref{sec4}--\ref{ScQMixEnv}.
Section \ref{sec4} contains auxiliary facts needed in our analysis.
In Section \ref{ScOcc} we obtain bounds on moderate deviation of occupation times
for diffusive walks. Occupation times play an important role in the analysis of transient walks
since traps which are behind several interesting phenomena related to one dimensional
RWRE can be conveniently described as sites with high expected occupation times (see e.g. \cite{DG2}).
Section \ref{ScLLTHit}
 contains the main ingredient of our analysis--the LLT for
the hitting times (Theorem \ref{LLThittingtime}) which is a new result and an
important fact in its own right. Following the approach of \cite{DG3} we deduce the LLT for the walker
position in Section \ref{sec5} and establish mixing of EVFP process in Section \ref{ScQMixEnv}.

Since our paper is devoted to the RW on a strip in i.i.d. random environment (see \eqref{EqC1}), several recent results on RWRE
are not discussed here. In particular we do not treat one dimensional RWRE in dependent environment.
A review of this subject can be found in \cite{DFS}. We just note that the first LLT in the quasiperiodic setting
is due to Sinai \cite{S2}. We also would like to mention a significant progress in the study of LLT for multidimensional
RWRE, in both transient \cite{BCR} and recurrent cases \cite{CD}. We note that while there are similarities
in the approaches of \cite{BCR, DG3} and the present paper, the details in multidimensional and one dimensional
cases are quite different. To a large extent, this difference is due to the fact that in dimension one in
the ballistic regime the walker could spend much longer time at an individual site than in similar conditions
in higher dimensions.

\section{RWRE on a strip: the model and review of results.}\label{sec8}

\subsection{Definition of the model.} The RWRE on a strip $\bbS\overset{\mathrm{def}}{=}\mathbb{Z} \times\{1,\ldots,m\}$
was introduced in \cite{BG1}. We shall now recall its definition. The set
$L_n\overset{\mathrm{def}}{=}\{(n,j):\,1\le j\le m\}\subset \bbS$ is called \textit{layer} $n$ of
the strip (or just layer $n$). The walker can jump from a site in $L_n$ only to a site
in $L_{n-1}$, $L_n$, or $L_{n+1}$. Let $\xi_t=(X_t,Y_t)$ be the coordinate of the walk at time
$t$,  $t=0,\,1,\,2,...$,
with $X_t\in \mathbb{Z}$, $1\le Y_t\le m$. An environment $\omega$ on
a strip is a sequence of triples of $m\times m$ matrices $\omega = \{(P_n, Q_n, R_n)\}_{n\in \mathbb{Z}}$
with non-negative matrix elements and such that
$P_n+Q_n+R_n$ is a stochastic matrix:
\begin{equation}
\label{stch}
(P_n+Q_n+R_n)\mathbf{1}=\mathbf{1},
\end{equation}
where $\mathbf{1}$ is a vector whose all components are equal to 1.
The transition kernel of the walk is given by
\begin{equation}\label{RWstrip}
\mathbb{P}_\omega(\xi_{t+1}=z'|\xi_t=z)=
\begin{cases}
P_n(i,j)& \text{ if } z=(n,i), z'= (n+1, j), \\
Q_n(i,j)&\text{ if } z=(n,i), z'=(n-1, j), \\
R_n(i,j)&\text{ if } z=(n,i), z'= (n, j)
\end{cases}
\end{equation}
Setting $\xi(0)=z_0\in L_0$ completes the definition of the Markov chain.
Let us introduce the relevant probability spaces. Let
$(\Omega,\mathcal{F},\mathbf{P})$ be the space of
random environments, and $(\mathfrak{X}_z, \mathcal{F}_{\mathfrak{X}_z},\Prob_{\omega,z})$ be
the space of the trajectories of the walk starting from $z\in S$;
$\mathrm{P}_z:=\bP\ltimes\Prob_{\omega,z}$ for the {\it annealed}
probability measure on $(\Omega\times\mathfrak{X}_z,\mathcal{F}\times \mathcal{F}_{\mathfrak{X}_z})$.
The expectations over $\Prob_{\omega,z}$, $\mathbf{P}$,
and $\mathrm{P}_z$ will be denoted by $\EXP_{\omega,z}$, $\mathbf{E}$, and $\mathrm{E}_z$ respectively.
We use concise notation $\Prob_\omega$ and $\mathrm{P}$
in place of $\Prob_{\omega,z_0}$~and~$\mathrm{P}_{z_0}$ when their meaning is obvious from the context.

The above model reduces to the simple RWRE on $\mathbb{Z}$ when $m=1$.
It also incorporates the RWRE on $\mathbb{Z}$ with bounded jumps.
The description of the corresponding reduction can be found in \cite{DG2}
(see also \cite{BG1}, \cite{G1}) and will not be repeated here.

In most cases we suppose that the following conditions are satisfied:
\begin{equation} \label{EqC1}
\{(P_n, Q_n, R_n)\}_{n\in \mathbb{Z}}  \text{ is an i.i.d. sequence,}
\end{equation}
\begin{equation}\label{EqC2*}
\begin{aligned}
&\text{There is $\varepsilon>0$ such that $\mathbf{P}$-almost surely for all $i,\,j\in[1,m]$}\\
& \left\|  R_n\right\| < 1-\varepsilon,\ \ ((I-R_n)^{-1}P_n)(i,j)>\varepsilon, \ \ ((I-R_n)^{-1}Q_n)(i,j)>\varepsilon,
\end{aligned}
\end{equation}
\begin{equation}\label{EqC3}
\begin{aligned}
&\text{There is a $\kappa>0$ such that $\mathbf{P}$-almost surely $R_n(i,i)\ge\kappa$ for all $i\in[1,m]$}.\\
\end{aligned}
\end{equation}

\subsection{Preparatory statements.} A fundamental role in the study of RWRE on a strip is played by
the \textit{moment Lyapunov exponents} $r(\alpha)$ which were
first introduced in \cite{G1}. In turn, defining $r(\alpha)$ requires introduction of
sequences of matrices $\psi_n$, $\zeta_n$, and $A_n$ which play a very important role
in many aspects of the analysis of RWREs on a strip.

Let $\psi_a$ be an $m\times m$ stochastic matrix. For $n>a$ define $\psi_{n,a}$ recursively:
\begin{equation}
\psi_{n,a}\overset{\mathrm{def}}{=}(I-R_{n}-Q_{n}\psi_{n-1, a})^{-1}P_{n}%
\label{EqPsi}%
\end{equation}
It is
easy to see that all $\psi_{n,a}$, $n\geq a$, are stochastic matrices (\cite[Lemma 2]{BG1}).

\begin{theorem}
(\cite{BG1})
\label{ThZeta} Suppose that condition \eqref{EqC2*} is satisfied.
Then

(a) For every sequence $\omega$ there exists
$\DS \zeta_{n}=\lim_{a\rightarrow-\infty}\psi_{n,a},$ 
where the convergence is uniform in $\psi_a$ and 
$\zeta_{n}$ does not depend on the choice of the sequence $\psi_a$.

(b) The sequence $\zeta_{n}=\zeta_{n}(\omega),\ -\infty<n<\infty,$
of $m\times m$ matrices is the unique sequence of stochastic
matrices which satisfies the following system of equations
\begin{equation}
\zeta_{n}=(I-Q_{n}\zeta_{n-1}-R_{n})^{-1}P_{n},\quad n\in\mathbb{Z}.%
\label{EqZeta}%
\end{equation}
%
\end{theorem}
Set
\begin{equation}
A_{n}\overset{\mathrm{def}}{=}(I-Q_{n}\zeta_{n-1}-R_{n})^{-1}Q_{n}.
\label{DefA}
\end{equation}
By the Multiplicative Ergodic Theorem the following limit exists a.e. and is independent of $\omega$
$$ \blambda=\lim_{n\to\infty} \frac{\ln \|A_{n-1}\dots A_0\|}{n}. $$

\begin{theorem}
(a) $\xi$ is recurrent iff $\blambda=0;$

(b) $\xi_n\to+\infty$ with probability one iff $\blambda<0;$

(c) $\xi_n\to-\infty$ with probability one iff $\blambda>0.$
\end{theorem}

\begin{lemma}
\label{Mainlemma} (\cite[Lemma 2]{G1})
Suppose that \eqref{EqC1} and \eqref{EqC2*} are satisfied.
Then the following limit exists and is finite for all $\alpha\ge0$
\begin{equation} \label{r2}
r(\alpha)=\lim_{n\to\infty}\left(\mathbf{E}||A_{n} \cdots
A_{1}||^\alpha\right)^{\frac{1}{n}}.
\end{equation}
The convergence in (\ref{r2}) is uniform in $\alpha\in[-\alpha_0 ,\,\alpha_0]$ for any $\alpha_0>0$.
Moreover
$$ r(0)=1, \quad  r'(0)=\blambda. $$
\end{lemma}

$r(\alpha)$ are called {\em moment Lyapunov exponents}.
\ignore{\begin{equation}
r(\alpha)\overset{\mathrm{def}}{=}
\limsup_{n\to\infty}\left(\mathbf{E}||A_{n} \cdots
A_{1}||^\alpha\right)^{\frac{1}{n}}.
\label{r}%
\end{equation}}

Let $(P, Q, R)$ satisfy \eqref{EqC2*}. By the foregoing discussion there exists a unique
$\zeta=\zeta(P, Q, R)$ such that $\zeta=(I-R-Q\zeta)^{-1} P.$
Let $\lambda(P, Q, R)$ be the leading eigenvalue of
$(I-R-Q\zeta(P, Q, R))^{-1} Q.$ We say that the
environment satisfies {\em non-arithmeticity condition} if the distribution of the random variable
$ \ln \lambda(P_0, Q_0, R_0)$ is non arithmetic.

Note that if $m=1$ then $\zeta_n=1$, $\DS A_{n}=\lambda(p_n, q_n)=\frac{q_n}{p_n}$, and
$\DS r(\alpha)=\mathbf{E}\left(\left(\frac{q_0}{p_0}\right)^\alpha\right)$.

\subsection{Annealed Limit Theorems for RWRE on a strip}
\label{SSAnn}

In Sections \ref{SSAnn}--\ref{SSQLT}, we consider the transient walks. To fix our notation we assume that
$\xi_n\to+\infty$ with probability one, that is $\blambda<0.$

For $s\in (0,2),$ let $\fL_s$ denote the cumulative distribution function of the sum
\begin{equation}
\label{ApprT}
 \ft=\sum_{n=1}^\infty \Theta_n (\Gamma_n-\epsilon_s)
\end{equation}
  where $\Theta_n$ is the Poisson process
on $\reals^+$ with the intensity $\DS \frac{s}{\theta^{s+1}},$ $\Gamma_n$ are i.i.d. exponential random
variables with parameter $1,$ $\{\Theta_n\}$ and $\{\Gamma_n\}$ are independent,
and $\eps_s=\begin{cases} 0, s<1 \\ 1, s\geq 1.\end{cases}$
It is not difficult to see (cf. \cite{DG1}) that $\ft$ has stable distribution of index $s.$
Let $\fF$ be the cumulative distribution function of the standard normal random variable.

Fix $z\in L_0.$ We assume from now on that, unless explicitly stated otherwise, the random
walk starts from~$z.$
Denote by $T_n=T(n\,|\,\omega,z)$ the hitting time of layer $n$ by the walk starting from $z.$
Below we suppose that
conditions \eqref{EqC1} and \eqref{EqC2*} are satisfied.

Let $s$ denote the positive solution of $r(s)=1$ where
$r$ is definied by \eqref{r2}.
Note that $s$ can be equal to $+\infty$
(e.g. if the walker has positive drift with
probability~1).

\begin{theorem}[The annealed limit theorem]\
\label{ThAnn}

(a) If $s<1$ and the non-arithmeticity condition holds then there is a constant $B$ such that
$$ \lim_{N\to \infty} \mathrm{P}\left(\frac{T_N}{B N^{1/s}}\leq t\right)=\fL_s(t), \quad
\lim_{N\to \infty} \mathrm{P}\left(\frac{X_N}{N^s}\leq t\right)=1-\fL_s\left((Bt)^{-1/s}\right).
 $$

(b) If $1<s<2$ and the non-arithmeticity condition holds then there are constants $B$ and $v$ such that
$$ \lim_{N\to \infty} \mathrm{P}\left(\frac{T_N-N/v}{B N^{1/s}}\leq t\right)=\fL_s(t), \quad
\lim_{N\to \infty} \mathrm{P}\left(\frac{X_N-Nv}{N^s}\leq t\right)=1-\fL_s\left(-B t v^{1+(1/s)}\right).
$$

(c) If $s>2$ then there are constants $B$ and $v$ such that
$$ \lim_{N\to \infty} \mathrm{P}\left(\frac{T_N-N/v}{B \sqrt{N}}\leq t\right)=\fF(t), \quad
\lim_{N\to \infty} \mathrm{P}\left(\frac{X_N-Nv}{N^s}\leq t\right)=1-\fF\left(-Bt v^{3/2}\right).
$$
\end{theorem}

For RWRE on $\integers,$ Theorem \ref{ThAnn} was obtained in \cite{KKS}.
For the case of the strip Theorem \ref{ThAnn}(c) was proven in \cite{R}.
In cases (a) and (b) the limit theorems for
$T_N$ were obtained in \cite{DG2}. The limit theorems for $X_N$ are simple corollaries of the
corresponding results for $T_N$ and the backtracking estimates of Lemma \ref{LmBack}, see \cite{KKS}.

\subsection{Quenched limit theorems.}
\label{SSQLT}
Theorem \ref{ThAnn}(c) shows that
the walk is diffusive iff $s>2.$  Let
\begin{equation}
\label{QDrift}
b_n=b_n(\omega)=\min(k: \EXP_\omega(T_k)\geq n).
\end{equation}

\ignore{Then there are constants
$c>0$ and $D>0$ such that for $\mathbf{P}$-a.e. environment $\omega$ the following holds:
\begin{align}\label{2.5}
&\lim_{n\rightarrow\infty}{n}^{-1}(T(n\,|\,\omega,z)-
\mathbb{E}_{\omega,z}{T(n\,|\,\omega,z)})=0
\ \hbox{ with $\mathbb{P}_{\omega,z}$-probability 1  and }\\
&\lim_{n\rightarrow\infty}{n}^{-1}
\mathbb{E}_{\omega,z}{T(n\,|\,\omega,z)}=c^{-1}<\infty, \label{2.5.0}
\end{align}
\begin{equation}\label{T-clt}
\lim_{n\rightarrow\infty}\mathbb{P}_{\omega,z}\left\{\frac{T(n\,|\,\omega,z)-
\mathbb{E}_{\omega,z} T(n\,|\,\omega,z)}{\sqrt{n}}<x\right\}
=\frac{1}{\sqrt{2\pi}D}\int_{-\infty}^x
e^{-\frac{u^2}{2 D^2}}du,
\end{equation}
\begin{equation}\label{llnforZ} \lim_{t\to\infty}t^{-1}Z_t= c>0 \text{
with $\Prob_{\omega,z}$-probability 1. }
\end{equation}

To apply Theorem \ref{Thllt}, we have to verify the asymptotic normality of $T_n$; the latter is the results of
the following lemma.}

\begin{theorem}
\label{LmTCLT}
If $s>2$ then there exists $\brD>0$ such that for almost every $\omega$ for all
$t\in \mathbb{R}.$
$$ \lim_{N\to\infty}\mathbb{P}_\omega\left(\frac{T_N-\mathbb{E}_\omega T_N}{\sqrt{N}\brD}<t\right)
=\fF(t),\quad
\lim_{N\to\infty}\mathbb{P}_\omega\left(\frac{X_N-b_N(\omega)}{\sqrt{N}v^{3/2} \brD}<t\right)
=\fF(t).
$$
\end{theorem}
Theorem \ref{LmTCLT} was proved in \cite{G} and independently in \cite{P} for the
simple RWRE on $\integers.$ It was proved for RWRE on a strip in \cite{G1}.

In contrast, for $s<2$ the quenched limit theorem fails. To describe the asymptotic
behaviour of the walks in this case we need the following notation. Let
$\fF_\Theta$ be the conditional distribution function of the sum \eqref{ApprT}
where the sequence $\{\Theta_n\}$ is fixed while $\Gamma$ are i.i.d. standard exponential random
variables.

\begin{theorem}
\label{ThQLT-ND}
Suppose that $s\in (0,2)\setminus\{1\}.$ Let
$\DS B_s=\begin{cases} 0 & \text{if } s<1\\ 1/v & \text{if } s>1 \end{cases}$
where $v$ is the speed of the walk (see Theorem \ref{ThAnn}(b)).

Then the process
$\DS t\to \Prob_\omega\left(\frac{T_N-B_s N}{N^{1/s}}\leq t\right)$ converges
in distribution as $N\to\infty$ to the process
$t\to \fF_\Theta(t)$ where $\{\Theta_n\}$ is the Poisson process on $[0, \infty)$ with
intensity $\frac{s}{\theta^{s+1}}.$
\end{theorem}

Theorem \ref{ThQLT-ND} was obtained independently in \cite{DG1, PS1, ESTZ} for the simple RWRE on
$\integers$ and extended to RWRE on a strip in \cite{DG2}. We note that one can also obtain
a limit of the joint distributions of $T_{\beta_1 N}, T_{\beta_2 N}, \dots, T_{\beta_k N}$ (\cite{PS2})
 by considering a Poisson process
$(u_n, \Theta_n)$ on $[0, \infty)\times [0, \infty)$ with the intensity $\frac{s}{\theta^{s+1}}$ and letting
$$ \ft_j=\sum_{n: u_n\leq \beta_j} \Theta_n (\Gamma_n-\eps_s). $$
This allows one to describe the quenched distribution of the walker's position using the fact that
$\DS \Prob_\omega(X_N\leq x)\approx \Prob_\omega(T_x\geq N),$
see \cite{PS2} for details.

\subsection{Local Limit theorems for a RWRE on a strip.}
Next we describe local limit theorems for RWRE on a strip.

\ignore{Suppose that
$\{ \cX_N\}$ is a sequence of integer valued
random variables which converge in distribution, that is there are a sequences
$B_N$ and $D_N$ and a piecewise differentiable distribution function $\fG$ such that
$$ \lim_{N\to\infty} \bP\left(\frac{\cX_N-B_N}{D_N}\leq t\right)=\fG(t). $$
We say that $\{\cX_N\}$ satisfies the Local Limit Theorem if for each $t$ such that
$\fg(t):=\fG'(t)$ exists and is positive the following holds. Let $z_N$ be a sequence such that
$\DS \frac{z_N-B_N}{D_N}=t$ then
$\DS  \lim_{N\to\infty} D_N \bP\left(\cX_N-B_N=z_N\right)=\fg(t). $

\begin{theorem}
\label{ThAnnLLT}
Suppose $s\neq 1, 2.$ Then $T_N$ and $X_N$ satisfy the LLT.
\end{theorem}
}

We begin with the quenched result. The LLT makes sense if the quenched limit of the
properly normalized hitting time exists. Thus we need to assume that $s>2.$

\begin{theorem} \label{LLThittingtime}
Suppose that $s>2.$
Then $\mathbf{P}$-almost surely there is a constant $\brD>0$ such that
\begin{equation} \notag
\sup_{k}\left|\brD\sqrt{2\pi n}\bbP_\omega(T(n\,|\,\omega,z)=k)-
 \exp\left(-\frac{(k-\mathbb{E}_{\omega,z} T_n)^2}{2\brD^2 n}\right)\right|
\to0\text{ as } n\to\infty .
\end{equation}

\end{theorem}

Next we discuss quenched LLT for the walkers position.
It turns out that in this case an additional local factor is needed in the LLT.
Denote
\begin{equation}\label{defrhoi}
\rho_{(k,i)}=\mathbb{E}_\omega\Card(n: \xi_n=(k,i)) \text{ and set }
a\de\mathbf{E}\left(\sum_{i=1}^m \rho_{(k,i)}\right).
\end{equation}

\begin{theorem}\label{ThQLLTstrip}
{\em (Quenched LLT)}
Suppose that conditions \eqref{EqC1}, \eqref{EqC2*}, \eqref{EqC3}, are satisfied
and $r(2)<1.$ Then there is a constant $D>0$ such that $\mathbf{P}$-almost surely the following holds.
For each $\eps, R>0$ there exists $n_0=n_0(\omega)$
such that for $n\geq n_0$ uniformly in $k$ satisfying
\begin{equation}
\label{NearBn1}
|k-b_n(\omega) |\leq R\sqrt{n}
\end{equation}
 we have
$$ \left|\frac{\sqrt{2\pi n} Da}{\rho_{(k,i)}} \exp\left[\frac{(k-b_n)^2}{2 D^2 n} \right]
\bbP_\omega(\xi_n=(k,i))-1\right|<\eps . $$
\end{theorem}

\begin{theorem}\label{ThAnnLLTs}
{\em (Annealed LLT)}
There is a $\mathbf{D}>0$ such that for each $\eps, R>0$ there exists $n_0$ such that for $n\geq n_0$
uniformly for
\begin{equation}
\label{NearNa}
\left|k-nv\right|\leq R\sqrt{n}
\end{equation}
 we have
$$ \left|\sqrt{2\pi n} \bD \exp\left[\frac{(k-nv)^2}{2 \bD^2 n} \right] \mathrm{P}(\xi_n=(k,i))-1\right|<\eps . $$
\end{theorem}

For the simple RWRE on $\integers,$ Theorems \ref{LLThittingtime}, \ref{ThQLLTstrip}
and \ref{ThAnnLLTs} are due to \cite{DG3}.
In this paper we extend this result to the strip using the approach of \cite{DG3}.

\ignore{The local limit theorem implies the annealed LLT and almost sure limit theorem for environment as seen by the particle.
\begin{theorem}
\label{ThAnnLLTs}
There is a $\mathbf{D}>0$ such that for each $\eps, R>0$ there exists $n_0=n_0(\omega)$ such that for $n\geq n_0$
uniformly for
\begin{equation}
\label{NearNa}
\left|k-\frac{n}{a}\right|\leq R\sqrt{n}
\end{equation}
 we have
$$ \left|\sqrt{2\pi n} \bD \exp\left[\frac{(k-\frac{n}{a})^2}{2 \bD^2 n} \right] \bP(X_n=(k,i))-1\right|<\eps . $$
\end{theorem}}

\subsection{Recurrent case.}
In this section we review the results for recurrent RWRE on the strip. Even though our paper does not
contain new results for recurrent walks, we believe it is useful to state them for the purpose of
comparison with the transient case.

For the simple RWRE on $\integers$ (that is with jumps to the nearest neighbors)
Sinai has proved \cite{S1} that the
walk exhibits the following asymptotic behaviour: there is a constant $D>0$ such that
$$\lim_{N\to\infty} \mathrm{P}\left(\frac{X_N}{D\ln^2 N}\leq t\right)=\fH(t). $$
This is what is called {\em the Sinai behavior}.
Here $\fH$ is the Kesten-Sinai distribution \cite{K-Rec} whose density is given by
$$ \fH'(t)=\frac{2}{\pi} \sum_{k=0}^\infty \frac{(-1)^k}{2k+1} \exp\left(-\frac{(2k+1)^2 \pi^2}{8} |t|\right). $$

\begin{theorem}
Consider recurrent RWRE on a strip and suppose that conditions \eqref{EqC1} and \eqref{EqC2*} are satisfied. Then either

(a) the walk exhibits the Sinai behavior; or

(b) there is a constant $K$ such that with probability 1 for all $n$
\begin{equation}
\label{BP}
 \frac{1}{K}\leq ||A_n\dots A_0||\leq K.
\end{equation}
Moreover, if \eqref{BP} holds, the environment is stationary and satisfies \eqref{EqC2*}
then there is a constant $D$ such that
$$ \bP\left(\frac{X_N}{D\sqrt{N}}\leq t\right)=\fF(t)
\quad\text{and for a.e. }\omega, \;
\Prob_\omega\left(\frac{X_N}{D\sqrt{N}}\leq t\right)=\fF(t).$$
\end{theorem}
The fact that the walk exhibits the Sinai behavior unless \eqref{BP} holds is due to \cite{BG2}.
The CLT for stationary environments satisfying \eqref{BP} was proved in \cite{DG4}.

\begin{theorem}(\cite{DG6})
  \label{ThQLLTstripRec}
  Suppose that \eqref{BP} holds.
Then
$\mathbf{P}$-almost surely the following holds. For each $\eps, R>0$ there exists $n_0=n_0(\omega)$
such that for $n\geq n_0$ uniformly for
$ \DS |k|\leq R\sqrt{n}$
 we have
$$ \left|\frac{\sqrt{2\pi n} Da}{\rho_{(k,i)}} \exp\left[\frac{k^2}{2 D^2 n} \right] \bbP_\omega(\xi_n=(k,i))-1\right|<\eps . $$
\end{theorem}

\subsection{Environment seen by the particle.}
\label{SSEnvByPart}
Let $\cT$ be the natural shift on the space of environments. Consider the Markov chain $(\omega^{(N)}, Y_N)$, where
$\omega^{(N)}=\cT^{X_N} \omega.$
This Markov chain is called \textit{environment seen by the particle} (ESP)
and it is an effective tool for studying RWRE.

Denote $\tilde{\Omega}=\Omega \times \{1,\dots, m\}$ the phase space of the Markov chain ESP.

\begin{theorem}\cite{DG5}
$(\omega^{(N)}, Y_N)$ admits an invariant measure $\bQ$ which is absolutely continuous
with respect to $\bP$ iff one of the following conditions is satisfied

either (a) the walk is transient and $s>1$ (that is the walk has positive speed);

or (b) \eqref{BP} holds. \newline In case (a) we have that
for every continuous function $\Phi:\tilde{\Omega} \to \reals$
$$\bQ(\Phi)=\frac{1}{a} \; \bE\left(\sum_{y=1}^m \Phi(\omega, y) \rho_{(0, y)} (\omega)\right)$$
where $\rho$ is given by \eqref{defrhoi}.
\end{theorem}

\begin{theorem}
\label{ThEPStrip}
(a) Suppose that either the walk is transient and $s>1$ or the walk is recurrent and \eqref{BP} holds.
Then
for every continuous function $\Phi:\tilde{\Omega} \to \reals$
$$ \lim_{N\to\infty} \mathrm{E}(\Phi(\omega^{(N)}, Y_N))=\bQ(\Phi). $$

(b) Suppose that either the walk is transient and $s>2$ or the walk is recurrent and \eqref{BP} holds.
Then for a.e. $\omega$ and for
every continuous function $\Phi:\tilde{\Omega} \to \reals$
\begin{equation}
\label{EqQEnv}
\lim_{N\to\infty} \EXP_\omega(\Phi(\omega^{(N)}, Y_N))=\bQ(\Phi).
\end{equation}
\end{theorem}

For recurrent RWRE both statements are due to \cite{DG6}. In the transient case, the result was
first proven for RWRE on $\integers$ in \cite{K-Ren} (part (a)) and \cite{L} (part (b)).
For the RWRE on the strip, part (a) is due to \cite{R}. Part (b) is new and will be proven in
Section \ref{ScQMixEnv} following the approach of \cite{DG3}.

\section{Preliminaries.}\label{sec4}
In the rest of the paper we give the proof of the local limit theorem for transient diffusive RWRE.
Therefore unless it is explicitly stated otherwise, we suppose that the walk is transient to the right
and that $s>2.$

\subsection{A Local limit theorem for independent summands.}
The following result from \cite{DMD} provides very general sufficient conditions
under which the Local Limit Theorem for sums of independent integer valued random variables holds.
\begin{theorem}(\cite{DMD})
\label{Thllt} Let $\eta_i,\ i\ge1,$ be independent integer valued random variables and let
$d_i=\sum_{j}\min[P(\eta_i=j),P(\eta_i=j+1)],$ $\mathfrak{d}_n=\sum_{i=1}^nd_i$. Denote
$\Xi_n=\sum_{i=1}^n\eta_i$. Suppose that there are numbers $\mathfrak{c}_n>0,$
$\mathfrak{h}_n$, $n\ge 1$, such that, as $n\to\infty$,  $\mathfrak{c}_n\to\infty$, $\limsup \mathfrak{c}_n^2/\mathfrak{d}_n<\infty$,
and $(\Xi_n-\mathfrak{h}_n)/\mathfrak{c}_n$ is asymptotically normal $\mathcal{N}(0,1)$. Then
\begin{equation}\label{llt}
\sup_{k}\left|\mathfrak{c}_n P\left(\Xi_n=k\right)-\frac{1}{\sqrt{2\pi}}
\exp\left(-\frac{(k-\mathfrak{h}_n)^2}{2\mathfrak{c}_n^2}\right)\right|\to 0 \text{ as }n\to\infty.
\end{equation}
\end{theorem}
\begin{remark} The requirement $\mathfrak{d}_n\to\infty$ implies that sufficiently many $d_i$'s
are positive. Had this not been the case, then it could happen that $\Xi_n$ would be taking,
say, only even values as $n$ becomes large. In our applications the  role of $\Xi_n$ is played by
$T_n$ and all the corresponding $d_i$'s are uniformly separated from $0$.
\end{remark}

\ignore{
Another setting where LLT holds is the following. Consider a sequence of Markov chains $\{\cZ_n(t)\}$
and choose
a site $z_n$ for each chain. Denone
$$ p_{1,n}=\bP(\cZ_n(t+1)=z_n|\cZ_n(t)=z_n), $$
$$ p_{2,n}=\bP(\exists s\in \naturals, \cZ_n(t+s)=z_n\text{ but }\cZ_n(t+1)\neq z_n |\cZ_n(t)=z_n), $$
$$ p_{3,n}=\bP(\forall s\in \naturals, \cZ_n(t+s)\neq z_n |\cZ_n(t)=z_n), $$
$$ q_n=\bP(z_n \text{ is visited}).$$
Let $\fl_n=\sum_{t=0}^\infty 1_{Z(t)=Z(t+1)=z_n}$ be the total number of times the walker spends two
consecutive moments of time at $z_n.$

\begin{lemma}
Suppose that there is a constant $\eta$ such that $p_{1,n}\geq \eta$ for all n and that
$$ \lim_{n\to\infty} p_{3,n}=0\quad \text{and}\quad \lim_{n\to\infty} q_n=1. $$

Then (a) $\fl_n \frac{p_{3,n}}{p_{3,n}+p_{1,n}}$ converge in law as $n\to\infty$ to standard exponential
random variable. Moreover the following local limit theorem holds: if
$\frac{p_{3,n}}{p_{3,n}+p_{1,n}}k_n\to u$ then
$$  \frac{p_{3,n}}{p_{3,n}+p_{1,n}} \bP(T_n=k_n)\to e^{-u}. $$

(b) For each $K$ there is $C$ such that for any pair $(k_{1,n}, k_{2,n})$
such that \\ $p_{3,n}|k_{1,n}-k_{2,n}|\leq K$ we have
$\frac{1}{C}\leq  \frac{\bP(\fl_n=k_{1,n})}{\bP(T_n=k_{2,n})}\leq C .$
Moreover for each $\eps$ there is $\kappa$
such that if $p_{3,n}|k_{1,n}-k_{2,n}|\leq \kappa $ then
$1-\eps\leq  \frac{\bP(\fl_n=k_{1,n})}{\bP(T_n=k_{2,n})}\leq 1+\eps.$
\end{lemma}

\begin{proof}
Both statements follow from the fact that \\
$\DS  \bP(\fl _n=k)=q_n \left(\frac{p_{1,n}}{p_{1,n}+p_{3,n}}\right)^k \frac{p_{3,n}}{p_{1,n}+p_{3,n}}. $
\end{proof}
}

\subsection{Occupation times.}

In the proofs of Theorems \ref{ThQLLTstrip}--\ref{ThEPStrip}  we assume that $s<\infty.$
Proofs become easier if $s=\infty$ and we leave the corresponding modifications
to the reader.

\begin{lemma}
\label{LmTail}
(a) There is a constant $C$ such that for each $(k,i)\in \bbS$
$$\bP\left(\rho_{(k,i)} >t\right)\leq C t^{-s}. $$
(b) For any $\hu>\frac{1}{s}$ for almost every $\omega$ there is a constant $C(\omega)$ such that
for each $(k,i)\in \bbS$
$$\rho_{(k,i)}<C(\omega) k^\hu.$$
\end{lemma}

\begin{remark}
In the case $s=\infty$ the statements of Lemma \ref{LmTail} read as follows.
For any $u, \hu>0$ we have that
$ \bP(\rho_{(k,i)}>t)\leq C t^{-u} \text{ and } \rho_{(k,i)}<C(\omega) k^{\hu}.$
\end{remark}
\begin{proof}
Part (a) follows from \cite[Lemma 3.3]{DG2}.
Part (b) follows from part (a) and the Borel-Cantelli Lemma.
\end{proof}
\subsection{Backtracking}
\begin{lemma}
\cite[Lemma 3.2]{DG2}
\label{LmBack}
Suppose that the walk is transient to the right. Then

\noindent
(a) There exists $C>0$ and $\theta<1$ such that
$$\mathrm{P}(X \text{ visits } k \text{ after } k+m)\leq C \theta^m. $$
(b) Accordingly, for almost every $\omega$ there is a constant $K(\omega)$ such that
$$\mathbb{P}_\omega(\exists k<n: X \text{ visits } k \text{ after } T_{k+\ln^2 n})\leq K(\omega) n^{-100}. $$
\end{lemma}

\ignore{
\begin{lemma} (\cite{G},Lemma 5)
\label{LmETkFluct}
There exists $\eps_0>0$ such that almost surely
$$ \lim_{n\to\infty} \frac{1}{\sqrt{n}} \max_{l\leq n^{\frac{1+\eps_0}{2}}}
\left| \mathbb{E}_\omega(T_{n+l}-T_n-la)\right|= 0. $$
\end{lemma}}

\section{Occupation times of large segments.}
\label{ScOcc}
Let $T(n\,|\,\omega,(k,j))$ be the hitting time of layer $n>k$ by the walk starting
from $(k,j)\in L_k$. Denote by $e_{k,n}$ the vector whose components are the expectations of
these hitting times:
\begin{equation}\label{4.12}
e_{k,n}(j)=\mathbb{E}_\omega
T(n\,|\,\omega,(k,j))
\end{equation}
\ignore{
\begin{equation}\label{4.12}
e(k)\de\left(
\begin{array}
[c]{l}%
e_k(1)\\
\vdots \\
e_k(m)\\%
\end{array}
\right), \ \ \hbox{where}\ \
e_k(j)=\mathbb{E}_\omega
T(n\,|\,\omega,(k,j))
\end{equation}}
As has been shown in \cite[formula (4.28)]{G1}
\begin{equation}\label{4.20}
e_{k,n} =\sum_{j=k}^{n-1}\zeta_{k}\ldots\zeta_{j}
\left(\sum_{i=0}^{\infty}H_j^{i}U_{j-i}\mathbf{1}\right).
\end{equation}
Here $U_j\de (I-Q_j\zeta_{j-1}-R_j)^{-1}$, $\mathbf{1}$ is a column vector whose
components are all equal to 1, we assume that $\zeta_{k}\ldots\zeta_{k}=I$,
\begin{equation}\label{notn}
H_j^i\de A_j\ldots A_{j-i+1},\ \hbox{ with the convention that }\
H_j^0=I,\ \ H_j^1=A_j
\end{equation}
and the matrices $A_n$ are defined by \eqref{DefA}.
Let $y_n$ be a sequence of $m$-dimensional probability vectors such that
$y_{n}=y_{n-1}\zeta_n$ for all $n\in\mathbb{Z}$. There is a unique sequence
satisfying these equations (Lemma 1 in \cite{G1}). The probabilistic meaning of
$y_0=(y_0(1),...,y_0(m))$ is the distribution of the starting point of the RW:
\begin{equation}\label{DistrofX0}
\mathbb{P}_{\omega}\{X(0)=(0,i)\}=y_0(i)
\end{equation}
(or, equivalently, the distribution
of the points in $L_0$ hit by the RW starting at $-\infty$).
In particular $T_n\de T(n\,|\,\omega,(0,\cdot))$ is the hitting time of $L_n$
by the walk starting from a random initial site in $L_0$ with the distribution of this site given by \eqref{DistrofX0}.

Let $\tau_j$ be the time the walk takes to reach $L_{j+1}$ after having reached $L_j$. Then
\begin{equation}\label{expect}
a_j\de\mathbb{E}_{\omega }(\tau_j)=y_j
\left(\sum_{i=0}^{\infty}H_j^{i}U_{j-i}\mathbf{1}\right)\ \text{ and we set }\ a\de \mathbf{E}(\mathbb{E}_{\omega }(\tau_0)).
\end{equation}

\begin{remark} The first relation in \eqref{expect} can be derived from
the identity $\tau_j={T}(j+1\,|\,\omega,z)-{T}(j\,|\,\omega,z)$ and \eqref{4.20}.
\end{remark}
\begin{remark}
\label{RmMeaningA}
It is easy to see that $\mathbf{E}(\mathbb{E}_{\omega }(\tau_0))=\mathbf{E}\left(\sum_{i=1}^m \rho_{(k,i)}\right)$
and therefore the definitions of $a$ given in \eqref{defrhoi} and \eqref{expect} are consistent. It is natural
and convenient to use
the two different interpretations of $a$ in our analysis.
\end{remark}

\begin{lemma}
\label{LmETkFluct1}
There exists $\eps_0>0$ such that almost surely
$$ \lim_{n\to\infty} \frac{1}{\sqrt{n}} \max_{|l|\leq n^{\frac{1+\eps_0}{2}}}
\left| \mathbb{E}_\omega(T_{n+l}-T_n-la)\right|= 0. $$
\end{lemma}
The following notations will be used in the rest of this section.
If $\mathcal{Y}:\Omega\mapsto\mathbb{R}$ is a random variable
then $||\mathcal{Y}||_p=\left(\mathbf{E}|\mathcal{Y}|^{p}\right)^\frac{1}{p}$, where $p=2+2\delta$ and
$\delta>0$ is such that $r(2+2\delta)<1$. Throughout the proof $\delta$ is fixed and
we write $||\cdot||$ for $||\cdot||_p$ whenever the difference between this norm
and the matrix norm is obvious from the context. Set
\begin{equation}\label{H1}
\mathcal{H}(n,\omega)=\sum_{j=0}^{n-1}(a_j-a),\ \
\mathcal{H}^*(n,\omega)=\max_{0\le s\le
n-1}\left|\sum_{j=0}^{s}(a_j-a)\right|.
\end{equation}
The maximal inequality of Lemma \ref{lem2}(a) below is the main ingredient of the proof.
We also state another maximal inequality (Lemma \ref{lem2}(b)) whose proof is quite similar
and which will be useful in Section \ref{ScQMixEnv}.

Let $W(P, Q, R)$ be any $\reals^m$ valued function of the triple of non-negative matrices
satisfying \eqref{stch} and let $w_n=W(P_n, Q_n, R_n).$
Let
\begin{equation}\label{expectgen}
\fa_j\de
 y_j
\left(\sum_{i=0}^{\infty}H_{j+i}^{i}U_{j}w_{j} \right)\ \text{ and }\ \fa\de \lim_{j\to\infty} \mathbf{E}(\fa_j).
\end{equation}
According to the analysis of \cite{DG2}[Section 3], there are constants $C>0, \theta<1$ such that with probability
at least $1-n^{-100}$
$$\left|\fa_j-\sum_{k=1}^m \rho_{(n,k)} w_n(k)\right|\leq C \theta^n.$$
Set
$$ \mathfrak{H}^*(n,\omega)=\max_{0\le s\le
n-1}\left|\sum_{j=0}^{s}(\fa_j-\fa)\right|. $$

\begin{lemma}\label{lem2} For any $\varepsilon>0$ there is a $C$ such that
the following relations hold:

$\DS (a) \quad
 ||\mathcal{H}^*(n)||\le Cn^{\frac{1}{2}+\varepsilon};\quad
$
$\DS (b) \quad
 ||\mathfrak{H}^*(n)||\le Cn^{\frac{1}{2}+\varepsilon}.
$

\end{lemma}

\begin{remark} The statements of this Lemma are similar to those of Lemma 4 in \cite{G}.
However, in the proof, an extra effort is needed in order to control the dependence between $a_j$'s
which in the case of a strip is much stronger than in the case of a simple walk.
\end{remark}

Once the maximal inequality of Lemma \ref{lem2}(a) is obtained, the proof of Lemma~\ref{LmETkFluct1}
proceeds exactly as the derivation of Lemma 5 from Lemma 4 in \cite{G}.

\begin{proof}[Proof of Lemma \ref{lem2}] We can suppose without any loss of rigor that $n^\varepsilon$
and $n^{1-\varepsilon}$ are integer numbers. One can present $\mathcal{H}(n)$ as
$$
\mathcal{H}(n)=\sum_{i=0}^{n^\varepsilon-1}\sum_{j=0}^{n^{1-\varepsilon}-1}(a_{i+n^\varepsilon j}-a).
$$
As will be seen below, the members of the inner sum are asymptotically independent random variables
(as $n\to\infty$) and this property plays a crucial role in the proof.
We can estimate $\mathcal{H}^*(n)$ as
$$
\mathcal{H}^*(n)\le \sum_{i=0}^{n^\varepsilon-1}\max_{0\le s\le
n^{1-\varepsilon}}\left|\sum_{j=0}^{s}(a_{i+n^\varepsilon j}-a)\right|.
$$
Hence
$$
||\mathcal{H}^*(n)||\le \sum_{i=0}^{n^\varepsilon-1}\left\Vert\max_{0\le s\le
n^{1-\varepsilon}}\left|\sum_{j=0}^{s}(a_{i+n^\varepsilon j}-a)\right|\,
\right\Vert=
n^\varepsilon\left\Vert\max_{0\le s\le
n^{1-\varepsilon}}\left|\sum_{j=0}^{s}(a_{n^\varepsilon j}-a)\right|\,\right\Vert,
$$
where the last equality is due to the fact the $a_j$ is a stationary sequence.
Let $\DS
\mathcal{S}(n, \varepsilon)=\max_{0\le s\le
n^{1-\varepsilon}}\left|\sum_{j=0}^{s}(a_{n^\varepsilon j}-a)\right|$ and present
$a_j-a=B(j,k)+D(j,k),$ where
\begin{equation}\label{B}
B(j,k)\de y_j\left(\sum_{i=0}^{k}H_j^{i}U_{j-i}\mathbf{1}\right)-
\mathbf{E}\left[y_j\left(\sum_{i=0}^{k}H_j^{i}U_{j-i}\mathbf{1}\right)\right],
\end{equation}
$$
D(j,k)\de y_j\left(\sum_{i=k+1}^{\infty}H_j^{i}U_{j-i}\mathbf{1}\right)-
\mathbf{E}\left[y_j\left(\sum_{i=k+1}^{\infty}H_j^{i}U_{j-i}\mathbf{1}\right)\right],
$$
and $k$ is to be specified later, see \eqref{SetK}.
We then have
$$
\begin{aligned}
\mathcal{S}(n, \varepsilon)&=\max_{0\le s \le
n^{1-\varepsilon}}\left|\sum_{j=0}^{s}(B(n^\varepsilon j,k)+D(n^\varepsilon j,k))\right|\\
&\le\max_{0\le s\le n^{1-\varepsilon}}\left|\sum_{j=0}^{s}B(n^\varepsilon j,k)\right|+
\sum_{j=0}^{n^{1-\varepsilon}-1}\left|D(n^\varepsilon j,k)\right| .
\end{aligned}
$$
Therefore
\begin{equation}\label{sum}
\begin{aligned}
||\mathcal{S}(n, \varepsilon)||&\le \left\Vert\max_{0\le s\le
n^{1-\varepsilon}}\left|\sum_{j=0}^{s}B(n^\varepsilon j,k)\right|\,\right\Vert+
\sum_{j=0}^{n^{1-\varepsilon}-1}||D(n^\varepsilon j,k)||\\
&=\left\Vert\max_{0\le s\le
n^{1-\varepsilon}}\left|\sum_{j=0}^{s}B(n^\varepsilon j,k)\right|\,\right\Vert+
n^{1-\varepsilon}||D(0,k)||,
\end{aligned}
\end{equation}
where the last equality is again due to stationarity. Since
by \eqref{r2}
$$\mathbf{E}(||H_j^i||^{2+2\delta})\le C r(2+2\delta)^i$$
we have for some constants $C_1,\ C_2,\ C_3$:
$$
\begin{aligned}
&||D(0,k)||=\left\Vert
y_j\left(\sum_{i=k+1}^{\infty}H_j^{i}U_{j-i}\mathbf{1}\right)-
\mathbf{E}\left[y_j\left(\sum_{i=k+1}^{\infty}H_j^{i}U_{j-i}\mathbf{1}\right)
\right] \right\Vert\\
&\label{25} \le C_1 \sum_{i=k+1}^{\infty}\left\Vert\,
\left\Vert H_j^{i}\right\Vert\,\right\Vert_p+
\mathbf{E}\left[y_j\left(\sum_{i=k+1}^{\infty}H_j^{i}U_{j-i}\mathbf{1}\right)\right]
\le C_2\sum_{i=k+1}^{\infty}r(2+2\delta)^{\frac{i}{2+2\delta}}=C_3 \beta^k,
\end{aligned}
$$
where $\beta\de r(2+2\delta)^{\frac{1}{2+2\delta}}<1$.

It remains to estimate
$\DS \mathcal{I}\de \max_{0\le s\le n^{1-\varepsilon}}
\left\Vert\sum_{j=0}^{s}\left|B(n^\varepsilon j,k)\right|\,\right\Vert$.
To this end we shall introduce independent random variables $\tilde{B}(n^\varepsilon j,k)$ such that for some $n_0$ the
following holds: if $n>n_0$ then
\begin{equation}\label{B1}
||B(n^\varepsilon j,k)-\tilde{B}(n^\varepsilon j,k)||\le\Const n^{-100}.
\end{equation}
This is done as follows.

For each $0\le j\le n^{1-\varepsilon}$ denote $I_j\de [n^\varepsilon j,\, n^\varepsilon (j+1)-1]$
and define stochastic matrices $\psi_s^{(j)}$ with $s\in I_j$
as follows: set $\psi_{n^\varepsilon j}^{(j)}(i_1,i_2)=m^{-1}$ and compute
$\psi_s^{(j)}$ for all $s>j n^\varepsilon$ as in \eqref{EqPsi}, namely
$\psi_s^{(j)}\de(I-R_{j}-Q_{j}\psi_{s-1}^{(j)})^{-1}P_{s}.$

Next, let $y_j^{(j)}=(m^{-1},\ldots,m^{-1})$ and compute $y_s^{(j)}=y_{s-1}^{(j)}\psi_s^{(j)}$
for $s>n^\varepsilon$.

Now define $\tilde{A}_{s}\de(I-Q_{s}\psi_{s-1}^{(j)}-R_{s})^{-1}Q_{s}$,
$\tilde{U}_s\de (I-Q_s\psi_{s-1}^{(j)}-R_s)^{-1}$,  and
$ \tilde{H}_s^i\de \tilde{A}_s\ldots \tilde{A}_{s-i+1}.$

Finally $\tilde{B}(j,k)$ is defined by \eqref{B} with the difference that the corresponding $A,\ U,\ H$
are replaced by the just defined $\tilde{A},\ \tilde{U},\ \tilde{H}$.

From now on we suppose that
\begin{equation}
\label{SetK}
k=K\ln n
\end{equation}
where $K$ is a large enough constant. It follows from the
above definitions that for large $n$ the random variables $\tilde{B}(n^\varepsilon j,k)$,
$0\le j\le n^{1-\epsilon}$, are independent and identically distributed simply because they are
(the same) functions of the parts of the environment belonging to pairwise disjoint boxes of the strip.
Also, $\mathbf{E}(\tilde{B}(n^\varepsilon j,k))=0$.

It follows from results obtained in \cite{DG2, G1} (see appendix in each of these papers) that
there is a $\bar{\beta}<1$ such that for $s\in I_j$
$\left|\left|||\zeta_s-\psi_{s}^{(j)}||\right|\right|_p\le C \bar{\beta}^{s-n^\varepsilon j}.$
The derivation of \eqref{B1} follows from this estimate in a standard way.

We can now estimate $\hat{\mathcal{S}}\de
\left\Vert\max_{0\le s\le
n^{1-\varepsilon}}\left|\sum_{j=0}^{s}B(n^\varepsilon j,k)\right|\,\right\Vert$.
Obviously
$$
\hat{\mathcal{S}}\le \left\Vert\max_{0\le s\le
n^{1-\varepsilon}}\left|\sum_{j=0}^{s}\tilde{B}(n^\varepsilon j,k)
\right|\,\right\Vert+
\left\Vert\sum_{j=0}^{s}|B(n^\varepsilon j,k)-\tilde{B}(n^\varepsilon j,k)|\,
\right\Vert
$$
By the Doob inequality for martingales we now have that
$$
\left\Vert\max_{0\le s\le
n^{1-\varepsilon}}\left|\sum_{j=0}^{s}\tilde{B}(n^\varepsilon j,k)
\right|\, \right\Vert\le
\frac{2+2\delta}{1+2\delta} \left\Vert\sum_{j=0}^{n^{1-\varepsilon}}\tilde{B}(n^\varepsilon j,k)\right\Vert
$$
and by the Marcinkiewicz-Zygmund inequality
$$
\left\Vert\sum_{j=0}^{n^{1-\varepsilon}}\tilde{B}(n^\varepsilon j,k)\right\Vert
\le C n^{(1-\varepsilon)/2}
$$
Part (a) of Lemma \ref{lem2}  is thus proved.

The proof of part (b) is similar. Namely we approximate
$\fa_j$ by $$\fa_j(k)\de  y_j
\left(\sum_{i=0}^{k}H_{j+i}^{i}U_{j}w_{j} \right) $$
and use weak dependence of $a_{j_1}(k)$ and $a_{j_2}(k)$ for $\left| j_1-j_2\right|>2k$
similarly to the argument of part (a).
\end{proof}

\section{Proof of Theorem \ref{LLThittingtime}.}
\label{ScLLTHit}
\begin{proof}[Proof of Theorem \ref{LLThittingtime}]
We shall use the construction of the enlarged random environment introduced in \cite{G1}.

For a RW starting from $z\in L_0$ present $T(n\,|\,\omega,z)$ as
\begin{equation}\label{equality}
{T}(n\,|\,\omega,z)=\tau_0+\tau_{1}+\ldots+
\tau_{n-1},
\end{equation}
where $\tau_j$ is the time the walk takes to reach $L_{j+1}$ after having reached $L_j$.

The random variables $\tau_j$ are not independent. However, they are conditionally independent,
where the condition is that the walk hits the layers $L_1,\ldots,L_{n-1}$ at a given sequence of points.
We shall now briefly describe how the measure $\mathbb{P}_{\omega,z}$ on $\mathfrak{X}_z$ can be presented
as an integral of the just mentioned conditional measures over a measure on the set of sequences in the strip.
The construction of this conditional measure is similar to the one used in \cite{G1}.

Let $ J=\{\,
\mathfrak{i}=\left((k,i_k)\right)_{0\le k<\infty}:\ 1\le i_k\le
m\  \}$ be the set of all
sequences of points in the strip $\bbS$ which contain exactly one point from each layer
and $ J_{i_0}\subset J$ be all such sequences starting with $(0,i_0)$.

Let $\zeta_k\equiv \zeta_k(\omega)$ be matrices from Theorem \ref{ThZeta}(a).
Note that due to
\eqref{EqC2*} there
is $\varepsilon>0$ such that the inequalities $\zeta_k(i,j)\ge\varepsilon$ hold for all $k,i,j$.

From now on, $i_0$ will be fixed and we shall define a Markov measure on $J_{i_0}$ in the usual way.
Namely, for $ k\ge 0$ consider a cylinder set in $J_{i_0}$:
\[
\mathrm{C}_{k}(i_0,\ldots,i_{k})\de\{\ \mathfrak{i}\in J_{i_0}\,:\,
\mathfrak{i}_j=(j,i_j)\ \hbox{ for all }\ j\in[0,k]\ \},
\]
where $\mathfrak{i}_j$ denotes the $j$-th coordinate of the sequence $\mathfrak{i}$. Set
\begin{equation}\label{3.2}
\Lambda_{\omega,i_0}(\mathrm{C}_{k}(i_0,\ldots,i_{k}))
\de \zeta_0(i_0,i_{1})
...\zeta_{k}(i_{k-1},i_{k}).
\end{equation}
$\Lambda_{\omega,i_0}$ defined by \eqref{3.2} is a Markov measure with transition probabilities
given by  stochastic matrices $\zeta_n.$ As usual, this measure can be
extended to the Borel sigma-algebra generated by the cylinder sets.
Let us denote this sigma-algebra by $\mathcal{F}_{i_0}$. We thus have a
probability space $(J_{i_0},\mathcal{F}_{i_0},\Lambda_{\omega,i_0}(d\mathfrak{i}))$.

We can now define a new probability space
$(\TO,\mathfrak{S},\mathcal{P})$ where
\[
\TO\de\Omega\times J_{i_0}=\{\,\tom=(\omega,\mathfrak{i})
\,:\, \omega\in\Omega,\ \mathfrak{i}\in J_{i_0}\,\}
\]
with a product sigma-algebra
$\mathfrak{S}\de\mathcal{F}\otimes\mathcal{F}_{i_0}$ and $\mathcal{P}(d\tom)\de \mathbf{P}(d\omega)
\Lambda_{\omega,i_0} (d\mathfrak{i})$ on $\TO$.

\smallskip\noindent\textbf{Definition.}
\textit{A pair $\tom=(\omega,\mathfrak{i})$ is called the enlarged
random environment. The set $\TO$ is the collection of
enlarged environments with
$(\TO,\mathfrak{S},\mathcal{P})$ being the corresponding probability space.}
\begin{remark}
The above construction depends on $i_0$ but the asymptotic behaviour of the walk doesn't
which is why in some of the notations above $i_0$ was dropped.
\end{remark}
Next, for $\mathfrak{i}\in J_{i_0}$ denote by $\mathfrak{X}^{\mathfrak{i}}_{i_0}$ the space  of trajectories  of
the walk which start at $(0,i_0)$ and reach each layer $L_j$, $j\ge 1$, for the first time at $(j,i_j)$.
For any $n\ge 0$ and $\xi\in \mathfrak{X}^{\mathfrak{i}}_{i_0}$ denote by $z_0,z_1,...,z_{T_n}$ the values of
the trajectory at times $t=0,1,...,T_n$, where, as usual, $T_n$ is the hitting time
of layer $n$ (with $\xi(0)=(0,i_0),\ \xi(T_n)=(n,i_n)$). The probability $V_{\omega,\mathfrak{i}}$ on the set of
trajectories from  $\mathfrak{X}^{\mathfrak{i}}_{i_0}$ is defined by
\begin{equation}\label{Vo}
V_{\omega,\mathfrak{i}}\left(\xi_t=z_t\ \forall\ t\in [0,T_n]\right)
\de\frac{1}{\prod_{k=0}^{n-1}\zeta_k(i_k,i_{k+1})}\prod_{t=0}^{T_n-1}
\mathbb{P}_\omega(\xi_{t+1}=z_{t+1}|\xi_t=z_t).
\end{equation}
It is known (\cite{BG1}, \cite{G1}) that if the RW is transient to the right then the probabilistic
meaning of matrices $\zeta_n$ is given by
\[
\zeta_{n}(i,j) = 
\mathbb{P}_{\omega}\left(\hbox{RW starting from $(n,i)$  hits $L_{n+1}$ at $(n+1,j)$}\right).
\]
Therefore $\Lambda_{\omega,i_0}(\mathrm{C}_{k}(i_0,\ldots,i_{k}))$ is the probability that a RW staring from
$(0,i_0)$ proceeds to $+\infty$ so that on its way it hits $L_j$ at $(j,i_j)$ for all $j\in[1,k]$.
We obtain that
\begin{equation}\label{decomposition}
\mathbb{P}_{\omega,(0,i_0)}(d\xi)=\int_{J_{i_0}}\Lambda_{\omega,i_0}(d\mathfrak{i})V_{\omega,\mathfrak{i}}(d\xi).
\end{equation}
Denote by $T(n|\omega,\mathfrak{i})$ the sum in the right hand side of \eqref{equality}
conditioned on the walk hitting each $L_j,\ j\ge0,$ at $(j,i_j)\in \mathfrak{i}$.

We shall now check that Theorem \ref{Thllt} implies that the LLT holds for $T(n|\omega,\mathfrak{i})$ for $\Lambda_{\omega,i_0}$--a.e.
$\mathfrak{i}$ and deduce from here
that the LLT holds for  $T(n|\omega,(0,i_0))$ by integrating over $\mathfrak{i}$ using \eqref{decomposition}.

Indeed, $T_n=T(n\,|\,\omega,\mathfrak{i})$ is a sum of independent random variables $\tau_j$ such that the corresponding
$\mathfrak{d}_n$ and $\mathfrak{c}_n$ (from Theorem \ref{Thllt}) grow linearly and
$\mathfrak{c}_n\mathfrak{d}_n^{-1}< \Const$. Also,
$T(n|\omega,\mathfrak{i})$ satisfies the CLT by the following result from \cite{G1}.
\begin{theorem}\label{Th5.1} Suppose that conditions \eqref{EqC1}, \eqref{EqC2*}, \eqref{EqC3}
are satisfied and $s>2.$
 Then for $\mathcal{P}$-a.e. enlarged
environment $\tom=(\omega,\mathfrak{i})$
\begin{equation}\label{T1-clt}
\lim_{n\rightarrow\infty}V_{\omega,\mathfrak{i}}\left\{
\frac{{T}(n\,|\,\omega,\mathfrak{i})-
E_{\omega,\mathfrak{i}}
{T}(n\,|\,\omega,\mathfrak{i})}{\sqrt{n}}<x\right\}
=\frac{1}{\sqrt{2\pi}\hat{\sigma}}\int_{-\infty}^x
e^{-\frac{u^2}{2{\hat{\sigma}}^2}}du,
\end{equation}
where $E_{\omega,\mathfrak{i}}$ is the expectation with respect to the measure $V_{\omega,\mathfrak{i}}$. Here
\begin{equation}\label{V-clt}
\hat{\sigma}^2=\lim_{n\rightarrow\infty}n^{-1}\mathrm{Var}_{\omega,\mathfrak{i}}
{T}(n\,|\,\omega,\mathfrak{i})
\end{equation}
and the convergence in (\ref{V-clt}) holds with
$\mathcal{P}$-probability 1.
\end{theorem}
Thus equation \eqref{llt} in our context reads as follows:
\begin{equation}\label{llt1}
\sqrt{2\pi n}\hat{\sigma}\ V_{\omega,\mathfrak{i}}\left(T_n=k\right)=
\exp\left(-\frac{(k-E_{\omega,\mathfrak{i}}T_n)^2}{2n\hat{\sigma}^2}\right)+
\varepsilon_n(k, \omega, \mathfrak{i})
\end{equation}
where for a.e. $(\omega, \mathfrak{i})$
$$\lim_{n\to\infty} \sup_{k\in\integers} \left|\varepsilon_n(k,\omega,\mathfrak{i}) \right|=0$$
and $\hat{\sigma}$ is the same as in Theorem \ref{Th5.1}.

It remains to carry out the integration of all parts of \eqref{llt1} over
$\Lambda_{\omega,i_0}(d\mathfrak{i})$. Obviously,
$$
\int_{J_{(0,i_0)}} V_{\omega,\mathfrak{i}}\left(T_n=k\right)\Lambda_{\omega,i_0}(d\mathfrak{i})=
\mathbb{P}_{\omega,(0,i_0)}\left(T_n=k\right).
$$
In order to control the RHS of \eqref{llt1}, note that
$$
\exp\left(-\frac{(k-E_{\omega,\mathfrak{i}}T_n)^2}{2n\hat{\sigma}^2}\right)=
\exp\left(-\frac{1}{2}\left((k-\mathbb{E}_{\omega,(0,i_0)}T_n)n^{-\frac{1}{2}}\hat{\sigma}^{-1}-\hat{\sigma}^{-1}\phi_n\right)^{2}\right)
$$
where $\phi_n=(E_{\omega,\mathfrak{i}}T_n-\mathbb{E}_{\omega,(0,i_0)}T_n )n^{-\frac{1}{2}}$.
By Theorem 7 in \cite{G1},
\begin{equation}
\label{CLTShift}
\phi_n
\text{ has asymptotically normal distribution }\mathcal{N}(0,\tilde{\sigma}^2),
\end{equation}
where
$\tilde{\sigma}$ is a function of the parameters of the model (see \cite{G1}, formula (3.5)).
It may happen that $ \tilde{\sigma}=0$ in which case $\phi_n\to 0$ as $n\to\infty$ and can therefore
be ignored. So, from now on we assume that $ \tilde{\sigma}>0$ (which is the only interesting case).

For a given $\varepsilon>0$, the family of functions $\exp(-(a-c^{-1}\phi)^2)$ with $c\ge \varepsilon $ and arbitrary $a$
is uniformly continuous in $\phi$. Therefore the usual property of convergence in distribution implies that
{\footnotesize
\begin{equation}
\label{NormConv}
\sup_k\left|\int_{J_{(0,i_0)}}\exp\left(-\frac{1}{2}\left(\frac{(k-E_{\omega}T_n)}{\sqrt{n}\hat{\sigma}}-\hat{\sigma}^{-1}\phi_n\right)^{2}\right)
\Lambda_{\omega,(0,i_0)}(d\mathfrak{i})- \exp\left(-\frac{(k-E_{\omega}T_n)^2}{2n(\hat{\sigma}^2+\tilde{\sigma}^2)}\right)\right|\to 0
\end{equation}}
as $n\to\infty$.

It remains to show that for $\bP$--almost every $\omega$
\begin{equation}
\label{IntErrror}
\int_{J_{(0,i_0)}} \varepsilon_n(k, \omega, \mathfrak{i})\Lambda_{\omega,(0,i_0)}(d\mathfrak{i})\to 0
\end{equation}
uniformly in $k$. To this end we note that due to \eqref{llt1}
\[
|\varepsilon_n|\le \sqrt{2\pi n}\hat{\sigma}\ V_{\omega,\mathfrak{i}}\left(T_n=k\right) +1.
\]
We claim that the first summand on the RHS is uniformly bounded.
Indeed, $T_n$ conditioned on $(\omega,\mathfrak{i})$ is a sum of independent random variables
$\tau_j.$ Moreover, due to \eqref{EqC3} there exists
$\epsilon>0$ such that for each $\mathfrak{i}$ for each $l$ the probability
$$V_{\omega,\mathfrak{i}}(\tau_j=l\,|\,\mathfrak{i})\leq 1-\epsilon$$
(in fact, one can take $\epsilon=\frac{1}{1+\kappa}$ where $\kappa$ is from
\eqref{EqC3}).
Now Theorem 3 from \cite[Part III, \S 2]{Pet} implies that there is a constant
$C$ such that
\begin{equation}
\label{QuenchConc}
\sqrt{2\pi n}\hat{\sigma}\ V_{\omega,\mathfrak{i}}\left(T_n=k\right)\le C.
\end{equation}
Therefore $|\varepsilon_n|$ is bounded above and \eqref{IntErrror}
follows from the
dominated convergence theorem.

Theorem \ref{LLThittingtime} is proved (with $\bar{D}^2=\hat{\sigma}^2+\tilde{\sigma}^2$)
where $\hat\sigma$ is from \eqref{llt1} and $\tilde\sigma$ is from
\eqref{CLTShift}.
\end{proof}

\section{LLT for the walker's position.}\label{sec5}
\subsection{The quenched LLT}
\begin{proof}[Proof of Theorem \ref{ThQLLTstrip}.]
  Take $\frac{1}{s}<u<\frac{1}{2}.$ Let $\ell=\ln^2 n,$ $\brk=k-\ell$.
  We claim that for $\mathbf{P}$-almost all $\omega$
\begin{equation}\label{return}
  \bbP_\omega(\exists k\leq n \;\;\exists m\in \naturals :
  X_m=k \text{ and } T_{\brk}<m-n^u)
\leq \frac{\brC(\omega)}{n^{100}}.
\end{equation}
Indeed, if $X_m=k$ and $m>T_{\brk}+n^u$
then one of the following events takes place:
\begin{align*} A_1=&\{X_t\in [k-2\ell, k+\ell]
\text{ for all } t\in [T_{\brk}, T_{\brk}+n^u]\},\\
A_2=&\{\exists t \in [T_{\brk}, T_{\brk}+n^u] \text{ such that } X_t<k-2\ell\}, \\
A_3=&\{\exists t \in [T_{\brk}, T_{\brk}+n^u] \text{ s. t. } X_t>k+\ell \text{ and then }
X \text{ backtracks to } k \}.
\end{align*}
$\bbP_\omega(A_2)$ and $\bbP_\omega(A_3)$ are $O(n^{-100})$ by Lemma \ref{LmBack}(b).
Take $\frac{1}{s}<u'<u''<u.$ If $A_1$ happens then there exists $k^*\in [k-2\ell, k+\ell]$
which is visited more than $n^{u''}$ times. However the  number of visits to $k^*$ has geometric
distribution with mean $\rho_\brk<C(\omega) n^{u'}.$ Thus $\bbP_\omega(A_1)\leq n^{-100}$ proving \eqref{return}.

Next we claim that given $R$ we can take $\brR$ so large that
if $k$ satisfies \eqref{NearBn1} then
\begin{equation}
\label{NTkClose}
|n-\EXP_\omega T_\brk|\leq \brR \sqrt{\brk}.
\end{equation}
Indeed
\begin{equation}
\label{NTk}
 n-\mathbb{E}_\omega T_\brk=
(n-\mathbb{E}_\omega T_{b_n})+(\mathbb{E}_\omega T_{b_n}-\mathbb{E}_\omega T_\brk).
\end{equation}
Observe that by definition
$ \mathbb{E}_\omega T_{b_n-1}<n\leq \mathbb{E}_\omega T_{b_n} $
and by Lemma \ref{LmETkFluct1}
$$\mathbb{E}_\omega (T_{b_n}-T_{b_n-1})=o(\sqrt{n})$$
so that the first term in \eqref{NTk} is $o(\sqrt{n}).$
Next, Lemma \ref{LmETkFluct1} also implies that
$$ \mathbb{E}_\omega(T_{b_n}-T_\brk)=a(b_n-\brk)+o(\sqrt{n}). $$
This implies \eqref{NTkClose} and shows moreover that
\begin{equation}
\label{NTDist}
 \frac{(n-\mathbb{E}_\omega T_\brk)^2}{k}\approx \frac{a^3 (b_n-\brk)^2}{n}.
\end{equation}

Thus for $j\in [0, n^u]$ we have due to Theorem \ref{LLThittingtime}
$$ \bbP_\omega(T_\brk=n-j)\approx \frac{1}{\sqrt{2\pi \brk}\brD} \exp-\left(\frac{(n-\mathbb{E}_\omega T_\brk)^2}{2\brD^2 \brk}\right).$$
On the other hand for $\bP$--almost every $\omega$ and all sufficiently large $n$
$$ \sum_{j=0}^{n^u} \bbP_\omega(\xi_j=(k,i)|X_0=\brk)=\rho_{(k,i)}+O\left(n^{-100}\rho_{(k,i)}\right)
=\rho_{(k,i)}+O\left(n^{-99}\right). $$
Thus
\begin{equation}
\label{PreLLT}
\bbP_\omega(\xi_n=(k,i))\approx
\frac{\rho_{(k,i)}}{\sqrt{2\pi \brk}\brD}
 \exp-\left(\frac{(n-\mathbb{E}_\omega T_\brk)^2}{2\brD^2 \brk}\right).
\end{equation}

Combining this with \eqref{NTDist}  we get
$$
\begin{aligned}
\bbP_\omega(\xi_n=(k,i))&\approx \frac{\sqrt{a} \rho_{(k,i)}}{\sqrt{2\pi  n}\brD} \exp-\left(\frac{(k-b_n)^2 a^3}{2\brD^2 n}\right)\\
&=\frac{ \rho_{(k,i)}}{\sqrt{2\pi  n} Da } \exp-\left(\frac{(k-b_n)^2 }{2 D^2 n}\right),
\end{aligned}
$$
where $D=\brD/a^{3/2}.$
\end{proof}

\subsection{The annealed LLT}
Given $k$ denote $\brk=k-\ln^2 n,$ $\hk=\frac{k+\brk}{2}.$

Let $\txi(n)$ be equal to $\xi(n)$ if the walker does not backtrack to $L_\brk$
during the time $[T_\hk, n]$
and $\txi(n)$ be the place of the first return to $L_\brk$ after $T_\hk$ otherwise.  In other words,
we stop our walk if it returns to $L_\brk$ after $T_\hk.$ By Lemma \ref{LmBack},
$\DS P(\txi(n)\neq \xi(n))=O\left(n^{-100}\right)$ and Theorem \ref{ThQLLTstrip} remains true with
$\xi$ replaced by $\txi.$ We need the following estimate.

\begin{lemma}
\label{LmConcPos}
The family of random variables
$$ \left\{\sqrt{n} \Prob_\omega(\txi(n)=z)\right\}_{n\in \naturals, z\in \bbS} $$
is uniformly integrable.
\end{lemma}

\begin{proof}
Let $I_j=\{l: l-\EXP_\omega(T_\brk)\in [j \sqrt{n}, (j+1)\sqrt{n}\}.$ Then
$$ \Prob_\omega(\txi(n)=z)=\sum_l \Prob_\omega(T_\brk=l) \Prob_\omega(\txi(n)=z|T_\brk=l) $$
$$ \leq \sum_j \max_{l\in I_j} \Prob_\omega(T_k=l)
\max_{y\in \{1\dots m\} }
\sum_{t=j\sqrt n}^{(j+1)\sqrt{n}} \Prob(\txi(t)=z| \txi(0)=(\brk, y)) $$
\begin{equation}
\label{SqrtNSum}
 \leq \sum_j \max_{l\in I_j} \Prob_\omega(T_k=l) \trho(z)
\end{equation}
where $\trho(z)$ is the maximum over $\hz\in L_\hk$ of the expected total number of
visits to site $z$ by the walk $\txi$ started from $\hz.$

Next we claim that there is a constant $C$ such that for all $\omega\in \Omega$ and
all $l\in I_j(\omega)$ we have
\begin{equation}
\label{TimeLDLoc}
\Prob_\omega(\txi(n)=l)\leq \frac{C\; \mathrm{Var}_\omega(T_\brk)}{j^2 n^{3/2}}
\end{equation}
Indeed, let $T'=T_{\brk/2},$ $T''=T_\brk- T_{\brk/2}.$ Clearly
$$ \Prob_\omega(\txi(n)=l)\leq $$
$$\Prob_\omega\left(\txi(n)=l, |T'-\EXP_\omega(T')|\geq \frac{j\sqrt{n}}{2}\right)+
\Prob_\omega\left(\txi(n)=l, |T''-\EXP_\omega(T'')|\geq \frac{j\sqrt{n}}{2}\right). $$
We will estimate the first term, the second one is similar.
$$\Prob_\omega\left(\txi(n)=l, |T'-\EXP_\omega(T')|\geq \frac{j\sqrt{n}}{2}\right)\leq $$
$$\Prob_\omega\left( |T'-\EXP_\omega(T')|\geq \frac{j\sqrt{n}}{2}\right)
\Prob_\omega\left(\txi(n)=l\big| |T'-\EXP_\omega(T')|\geq \frac{j\sqrt{n}}{2}\right). $$
The first factor is smaller than
$$ \frac{4 \mathrm{Var}_\omega(T')}{j^2 n}\leq \frac{4 \mathrm{Var}_\omega(T_\brk)}{j^2 n} $$
due to Chebyshev inequality. On the other hand the
proof of \eqref{QuenchConc} shows that the second  factor  is smaller than $\frac{C}{\sqrt{n}}$
proving \eqref{TimeLDLoc}.

Rewriting \eqref{TimeLDLoc} as
$$ \max_{l\in I_j(\omega)} \Prob_\omega(T_\brk=l)\leq \frac{C \; \mathrm{Var}_\omega(T_k)}{j^2 n^{3/2}} $$
and summing over $j$ in \eqref{SqrtNSum} we obtain that
$$ \sqrt{n} \Prob_\omega(\txi(n)=z)\leq C\; \frac{\mathrm{Var}_\omega(T_\brk)}{n} \trho(z) . $$
Now the uniform integrability of the LHS follows from the following facts:

(a) $ \DS \frac{\mathrm{Var}_\omega(T_\brk)}{n}$ and $ \trho(z)  $ are independent (since
the first variable depends
only on environment to the left of $\brk$
the second variable depends
only on environment to the right of $\brk$);

(b) $ \DS \frac{\mathrm{Var}_\omega(T_\brk)}{n}$ is uniformly integrable as follows easily from the explicit
expression for $\mathrm{Var}_\omega(T_\brk)$ given in \cite[Equation (4.28)]{G1};

(c) $\{\trho(z)\}_{z\in \bbS}$ are uniformly integrable since $\trho(z)\leq C \rho(z)$ and
$\{\trho(z)\}_{z\in \bbS}$ are uniformly integrable due to Lemma \ref{LmTail}(a).
\end{proof}

\begin{proof}[Proof of Theorem \ref{ThAnnLLTs}]
Once the uniform integrability of $\{\sqrt{n} \Prob_\omega(\txi(n)=z)\}$ is established,
the derivation of  Theorem \ref{ThAnnLLTs} from Theorem \ref{ThQLLTstrip}
is similar to derivation of Theorem \ref{LLThittingtime} from the LLT in the enlarged environment
which is explained in Section \ref{ScLLTHit} so we just sketch the argument
leaving the details to the reader. The proof consists of the following steps.

(I) We have
$$ \mathrm{P}(\xi_n=(k,i))=\bE(\Prob_\omega(\xi_n=(k,i))). $$
Theorem \ref{ThQLLTstrip} and Lemma \ref{LmConcPos} allow us to replace the above expectation by
$$ \bE\left(\frac{\rho_{(k,i)}}{\sqrt{2\pi n} Da} \exp-\left[\frac{(k-b_n(\omega))^2}{2 D^2 n} \right] \right) .$$

(II) Asymptotic independence of $b_n$ and $\rho_{k, i}$ (which comes from the fact that
$\rho_{(k, i)}$ can be well approximated by a variable which depends only on the environment
to the right of $n-\ln^2 n$ while $b_n$ can be
well approximated by a variable which depends only on the environment from the left of
$n-\ln^2 n$) allows us to replace the last expectation by the product
\begin{equation}
\label{Product}
 \bE\left(\frac{\rho_{(k,i)}}{a} \right)
\bE\left(\frac{1}{\sqrt{2\pi n} D}  \exp-\left[\frac{(k-b_n(\omega))^2}{2 D^2 n} \right] \right) .
\end{equation}

(III) The first factor in \eqref{Product} equals to 1 due to Remark \ref{RmMeaningA}.

(IV) Asymptotic normality of $b_n(\omega)$ shows that the second factor in \eqref{Product}
is asymptotic to
$$ \bE\left(\frac{1}{\sqrt{2\pi n} \mathbf{D}}  \exp-\left[\frac{(k-b_n(\omega))^2}{2 \mathbf{D}^2 n} \right] \right)
$$
where $\mathbf{D^2}=D^2+\hD^2$ and $\hD$ is the limiting variance of
$\DS \frac{b_n(\omega)}{\sqrt{n}}$
(cf. the derivation of \eqref{NormConv} in Section \ref{ScLLTHit}).
\end{proof}

\section{Mixing for environment seen by the particle.}
\label{ScQMixEnv}
\begin{proof}[Proof of Theorem \ref{ThEPStrip}(b)]
As has been explained in Section \ref{SSEnvByPart}, the result is known in the recurrent case, so it remains to consider
the transient case.

Denote $\omega_n\de (P_n, Q_n, R_n).$
It suffices to prove \eqref{EqQEnv}
for a dense set of functions, in particular, it is
enough to consider the case when
$\Phi(\omega, k)$ depends only on $\omega_n$ for $|n|\leq M$ for some finite $M.$
Below we shall give the proof in the case $M=0.$ The case of arbitrary finite $n$ requires routine modifications
which are left to the reader. Thus $\Phi=\Phi(\omega_0, Y).$ Let $w_n$ be the vector with
components $w_n(k) =\Phi(\omega_n, k).$
We have
$$ \EXP_\omega\left(\Phi\left(\omega^{(N)}, Y_N\right)\right)=
\sum_{k=1}^m \sum_{n=-\infty}^\infty \Prob_\omega(X_N=n, Y_N=k) w_n(k) . $$
By the quenched CLT for each $\eps$ there exists $R$ such that for all sufficiently large
$N,$ the sum over $n$ such that $|n-b_N|\geq R\sqrt{N}$ is less than $\eps$
where $b_N(\omega)$ is defined by \eqref{QDrift}.

Next, fix small constants $\delta_1, \delta_2$ and let
$$N_l=l^{2-\delta_1}, \quad I_{j, l}=
\left[N_l+j N_l^{\frac{1}{2}-\delta_2}, \; N_l+(j+1) N_l^{\frac{1}{2}-\delta_2}\right] $$
and let $x_{l,j}$ be the center of $I_{l,j}.$
Let $$\mathfrak{r}_{l,j} =\sum_{n\in I_{l,j}} \sum_{k=1}^m \rho_{n, k} w_n(k)=
\sum_{n\in I_{l,j}} \sum_{k=1}^m \rho_{n, k} w_n(k)+\cO\left(\theta^{N_l}\right).
 $$
 Lemma \ref{lem2} shows that
$$ \bP\left(\left| \frac{\mathfrak{r}_{l, j}}{N_l^{1/2-\delta_2}}-\fa\right|>\eps \right)\leq C
\frac{1}{\eps^{2+\delta} N_l^{(1/2-\delta_2)(1+\delta/2)}}
=\frac{C}{\eps^{2+\delta} l^{(1/2-\delta_2)(1+\delta/2)(2-\delta_1)}}
. $$
Thus if
$$\left(\frac{1}{2}-\delta_2\right)\left(1+\frac{\delta}{2}\right)\left(2-\delta_1\right)-2(2-\delta_1) \delta_2>1 $$
then
$$ \sum_{l=1}^\infty \sum_{|j|<N_l^{2\delta_2}}
\bP\left(\left| \frac{\mathfrak{r}_{l, j}}{N_l^{1/2-\delta_2}}-\fa\right|>\eps \right)<\infty, $$
so Borel--Cantelli Lemma gives that
\begin{equation}
\label{EnvErg}
\frac{\mathfrak{r}_{l,j}}{N_l^{1/2-\delta_2}}\to \fa
\end{equation}
uniformly for $|j|<N_l^{2\delta_2}.$ Given $N$ let $l$ be such that $N_l\leq N<N_{l+1}.$
Now the quenched LLT (Theorem \ref{ThQLLTstripRec}) gives that
$$ \sum_{n=b_N-R\sqrt{N}}^{b_N+R\sqrt{N}} \sum_{k=1}^m \Prob_\omega\left(\xi_N=(n,k)\right) w_n(k)
\approx
\sum_j \exp\left[-\frac{\left(x_{l,j}-b_N\right)^2}{2 D^2 N}\right]
\sum_{n\in I_{l,j}} \sum_{k=1}^m
\frac{\rho_{(n, k)}} {\sqrt{2\pi N} Da} w_n(k)
$$
where the first summation is over $j$ such that
$\DS  I_{l,j}\cap \left[b_N-R\sqrt{N}, b_N+R\sqrt{N}\right] \neq\emptyset$
and "$\approx$" holds since
$\exp\left[-\frac{\left(x-b_N\right)^2}{2 D^2 N}\right]$ is approximately constant when $x$ varies over $I_{l,j}.$

Next, \eqref{EnvErg} allows us to replace the sums over $n$ and $k$ by
$N_{l}^{1/2-\delta_2} \fa .$ Performing summation over $j$ we obtain the Riemann sum for
$$ \fa \int_{-R}^{R} \; \frac{1}{\sqrt{2\pi} D} e^{-x^2/2D^2} dx .$$
Since $R$ can be chosen arbitrarily large, \eqref{EqQEnv}  follows.
\end{proof}

\end{document}